\documentclass[11pt]{article}

\usepackage{amsmath,amssymb,amsthm,mathtools}
\usepackage{bm}
\usepackage{booktabs}
\usepackage{array}
\usepackage{enumitem}
\usepackage{graphicx}
\usepackage[hidelinks]{hyperref}
\usepackage[a4paper,left=2.5cm,right=2.5cm,top=3cm,bottom=3cm]{geometry}

\newtheorem{theorem}{Theorem}[section]
\newtheorem{proposition}[theorem]{Proposition}
\newtheorem{lemma}[theorem]{Lemma}
\newtheorem{remark}[theorem]{Remark}

\newcommand{\R}{\mathbb{R}}
\newcommand{\E}{\mathbb{E}}

\newcommand{\D}{\mathrm{d}}
\newcommand{\Var}{\operatorname{Var}}
\newcommand{\tr}{\operatorname{tr}}

\newcommand{\edist}{\stackrel{{\cal D}}{=}}
\newcommand{\vertk}{\stackrel{{\cal D}}{\rightarrow}}

\title{Four-Moment Approximations and Fast $p$-Values
for BHEP Tests of Multivariate Normality}

\author{Bruno Ebner \and Norbert Henze}

\date{}

\author{Bruno Ebner\footnote{Institute of Stochastics, Karlsruhe Institute of Technology (KIT), 76131 Karlsruhe, Germany. e-mail: bruno.ebner@kit.edu} \and Norbert Henze\footnote{Institute of Stochastics, Karlsruhe Institute of Technology (KIT), 76131 Karlsruhe, Germany. e-mail:
henze@kit.edu}}

\begin{document}

\maketitle

\begin{abstract}
The Baringhaus--Henze--Epps--Pulley (BHEP) tests form a widely applicable class of affine invariant and consistent tests for multivariate normality.  Under the null hypothesis, the test statistic converges to a weighted sum of independent chi-squared random variables. Although closed-form expressions for the first three cumulants of this limiting distribution are known for arbitrary dimension $d$ and smoothing parameter $\beta$, an explicit expression for the fourth cumulant has so far been available only in a special univariate case. We derive the fourth cumulant for arbitrary $d\geq1$ and $\beta>0$ and use the resulting first four moments to construct Johnson and Pearson approximations to the limiting null distribution. These yield essentially instantaneous analytic approximations to BHEP $p$-values, without numerical eigenvalue calculations. Using the recently obtained complete spectrum as a benchmark, we show that both four-moment approximations are highly accurate over a broad range of dimensions, smoothing parameters and upper-tail probabilities, and substantially improve on two- and three-parameter lognormal approximations. Monte Carlo results indicate good finite-sample calibration for many parameter combinations, although convergence to the limiting distribution may be slow in higher dimensions for extreme values of the smoothing parameter.
\end{abstract}

\medskip
\noindent
\textbf{Keywords:}
BHEP test; Johnson system; Pearson system; multivariate normality; quadratic Gaussian
functional; cumulants; moment matching; $p$-value approximation.

\medskip
\noindent
\textbf{2020 Mathematics Subject Classification.}
Primary 62H15; Secondary 62E20, 62E17.

\section{Introduction}

Tests of multivariate normality based on the empirical characteristic
function have several attractive features. In particular, the
Baringhaus--Henze--Epps--Pulley (BHEP) class of tests is affine
invariant, consistent against every fixed nonnormal alternative, and
applicable in arbitrary dimension and for arbitrary sample size,
subject only to nonsingularity of the sample covariance matrix (see
\cite{baringhaushenze1988,ebnerhenze20,eppspulley1983,henze2002,henzewagner1997,henzezirkler1990}).

Let $X_1,\ldots,X_n$ be independent copies of a $d$-dimensional random
(column) vector, and let $\overline X_n=\frac1n\sum_{j=1}^n X_j$ and
$ S_n=n^{-1} \sum_{j=1}^n   (X_j-\overline X_n)(X_j-\overline X_n)^\top
$ denote the sample mean and empirical covariance matrix, respectively.
We assume throughout that $n \ge d+1$ and that the distribution of $X_1$ is 
absolutely continuous with respect to Lebesgue measure, which entails that $S_n$ is
nonsingular with probability one (see, e.g., \cite{eatonperl73}).
Put
\[
 Y_j=S_n^{-1/2}(X_j-\overline X_n),
 \qquad j=1,\ldots,n,
\]
where $S_n^{-1/2}$ denotes the symmetric positive definite square root
of $S_n^{-1}$.

For $\beta>0$, let
\[
 \varphi_\beta(t)
 =  (2\pi\beta^2)^{-d/2}
 \exp\left(-\frac{\|t\|^2}{2\beta^2}\right),
 \qquad t\in\R^d,
\]
where $\| \cdot \|$ denotes the Euclidean norm on $\mathbb{R}^d$. 

The BHEP statistic depends on a smoothing parameter $\beta>0$ and is
based on a weighted $L^2$-distance between the empirical characteristic
function of suitably standardized observations and the characteristic
function of the standard $d$-variate normal distribution. Under the
hypothesis of normality, its limiting distribution is a quadratic
Gaussian functional and, equivalently, a weighted sum of independent
chi-squared random variables. A complete characterization of the
corresponding eigenvalues has only recently become available; some of
them are determined as roots of transcendental equations and of a
Fredholm determinant; see \cite{ebneretal26}. The precise definitions and the
spectral representation are recalled in Section~\ref{secfirstfour}.

A statistical test requires either critical values or,
preferably, a $p$-value corresponding to the observed value of
the BHEP statistic. Simulation provides one possibility, but it must either
be performed repeatedly or be replaced by extensive tables. An
analytic approximation to the null distribution is therefore
desirable.

Henze and Wagner \cite{henzewagner1997} derived explicit expressions
for the expectation, variance and third central moment of
the limiting null distribution for arbitrary $d$ and $\beta$. In the particular case
$d=1$ and $\beta=1$, Henze \cite{henze1990} had previously obtained
the first four moments and used them to fit distributions from the
Johnson and Pearson systems. The Johnson approximation provided
accurate upper quantiles and led to a simple normalizing
transformation by which approximate $p$-values could be computed.

The purpose of the present paper is to extend this idea to the full
BHEP family. Our contributions are threefold. First, we derive the
fourth cumulant of the limiting null distribution in closed form for arbitrary
$d\geq1$ and $\beta>0$. Second, we use the resulting first four
moments to construct Johnson and Pearson approximations to the limiting null
distribution and derive  direct analytic approximations to the
$p$-values of the BHEP statistic. Third, we investigate numerically
the accuracy of these $p$-values, both for the limiting distribution
and for finite sample sizes.

The explicit fourth moment obtained here turns the available analytic information about the BHEP limit law into a practical method for computing approximate $p$-values. The recent complete spectral characterization of the BHEP covariance operator also makes it possible to compute highly accurate reference values for the limiting distribution. This provides a natural benchmark for assessing the accuracy of the moment-based approximations proposed here.

The remainder of the paper is organized as follows.
Section~\ref{secfirstfour} reviews the limiting null distribution of the BHEP statistic and the first three cumulants, and derives the fourth cumulant in closed
form. Section~\ref{rem:spectral_power_sum} develops the Johnson and Pearson four-moment approximations and recalls the lognormal benchmark approximations. 
Section~\ref{sec:limit_accuracy} assesses the accuracy of these approximations using a highly accurate spectral reference distribution. Section~\ref{sec:finite_sample} investigates their
finite-sample calibration by Monte Carlo simulation, and Section~\ref{sec:discussion} concludes with a discussion. The implementation is described in the software availability statement, and the proof of the fourth-cumulant
formula is given in Appendix~A.

\section{The BHEP limit law and first four cumulants}
\label{secfirstfour}

With the notation introduced above, let
\[
\psi_n(t) = \frac{1}{n}\sum_{j=1}^n \exp\bigl({\rm i}t^\top Y_j\bigr),
\qquad t\in\mathbb{R}^d, 
\]
where ${\rm i}$ denotes the imaginary unit. Define
\[
T_{n,\beta} = n\int_{\mathbb{R}^d} \left| \psi_n(t)-\exp(-\|t\|^2/2) \right|^2
\varphi_\beta(t)\, \D t.
\]
Writing $\vertk$ for convergence in distribution and $\edist$ for equality in distribution, under the hypothesis that the observations have a nondegenerate
$d$-variate normal distribution,
\[
T_{n,\beta}\vertk T_\beta(d),
\]
where
\[
 T_\beta(d)  =  \int_{\R^d}Z^2(t)\varphi_\beta(t)\, \D t
\]
and $Z$ is a centered Gaussian process with covariance kernel
\begin{equation*}
 K(s,t)
 =  \exp\left(-\frac{\|s-t\|^2}{2}\right) -  \left\{  1+s^\top t+\frac{(s^\top t)^2}{2}  \right\}
 \exp\left( -\frac{\|s\|^2+\|t\|^2}{2}  \right).
\end{equation*}
Let $A=A_{\beta,d}$ denote the integral operator on $L^2(\mathbb{R}^d,\varphi_\beta(t) \D t)$ defined by
\[
 (Aq)(s)  =  \int_{\R^d}  K(s,t)q(t)\varphi_\beta(t)\, \D t.
\]
If $(\lambda_j)_{j\geq1}$ are its nonzero eigenvalues, then
\begin{equation}\label{eq:chisquare_rep}
T_\beta(d) \ \edist \ \sum_{j\ge 1} \lambda_j N_j^2,     
\end{equation}
where $N_1, N_2,\ldots$ are independent standard normal random variables. Next, for $r\ge 1$, let $\kappa_r=\kappa_r(d,\beta)$ denote the $r$th
cumulant of $T_\beta(d)$.

\begin{proposition}\label{prop:cumulants}
For each integer $r\geq1$,
\begin{equation}
 \kappa_r  =  2^{r-1}(r-1)!  \sum_{j\geq1}\lambda_j^r
 =  2^{r-1}(r-1)!\tr(A^r),
\label{eq:cumulants_trace}
\end{equation}
where ${\rm tr}$ denotes trace. Consequently,
\begin{align}
 \kappa_r
 &=  2^{r-1}(r-1)!  \int_{(\R^d)^r}  K(t_1,t_2)K(t_2,t_3)\cdots K(t_r,t_1)
 \nonumber\\ 
 &\hspace{35mm}\times  \prod_{\ell=1}^r\varphi_\beta(t_\ell)
 \,\D t_1\cdots \D t_r. 
\label{eq:cumulants_integral}
\end{align}
\end{proposition}

\smallskip
\begin{proof}
Representation \eqref{eq:chisquare_rep} and independence of the $N_j$ yield
\[
 \kappa_r  =  \sum_{j\ge 1}  \kappa_r(\lambda_jN_j^2).
\]
Since the $r$th cumulant of a chi-squared random variable with one
degree of freedom equals $2^{r-1}(r-1)!$,
the first equality in \eqref{eq:cumulants_trace} follows. The trace
representation and \eqref{eq:cumulants_integral} are standard
consequences of the integral representation of $A$.
\end{proof}
The first three cumulants are $\kappa_1=\E[T_\beta(d)]$,
 $\kappa_2=\Var(T_\beta(d))$, and
\[
 \kappa_3
 =  \E\left[  \{T_\beta(d)-\E T_\beta(d)\}^3  \right].
\]
For completeness, we recall the explicit expressions for these
cumulants obtained in \cite{henzewagner1997}. To this end,
\[
 a_\beta=1+2\beta^2,\qquad
 b_\beta=1+4\beta^2+3\beta^4,\qquad
 c_\beta=1+4\beta^2+2\beta^4.
\]
Then
\begin{align*}
\kappa_1(d,\beta)
&= 1-a_\beta^{-d/2} \left[ 1+\frac{d\beta^2}{a_\beta} +\frac{d(d+2)\beta^4}{2a_\beta^2}
\right],
\\[1mm]
\kappa_2(d,\beta)
&= 2(1+4\beta^2)^{-d/2} + 2a_\beta^{-d} \left[
1+\frac{2d\beta^4}{a_\beta^2} +\frac{3d(d+2)\beta^8}{4a_\beta^4} \right]
\nonumber\\
&\quad
- 4b_\beta^{-d/2}  \left[1+\frac{3d\beta^4}{2b_\beta}
+\frac{d(d+2)\beta^8}{2b_\beta^2} \right],
\\[1mm]
\kappa_3(d,\beta)
&= 8(1+3\beta^2)^{-d} \nonumber\\
&\quad
- 12(a_\beta c_\beta)^{-d/2} \left[ 2+\frac{d\beta^4}{a_\beta^2}
+\frac{2d\beta^6}{a_\beta c_\beta} +\frac{2d\beta^8}{a_\beta^2c_\beta}
+\frac{d(d+2)\beta^{12}}{a_\beta^2c_\beta^2} \right]
\nonumber\\
&\quad
+ 6(a_\beta b_\beta)^{-d/2} \left[ 4+\frac{4d\beta^4}{b_\beta}
+\frac{4d\beta^6}{a_\beta b_\beta} +\frac{d(d\! +\! 2)\beta^8}{a_\beta^2b_\beta}
+\frac{3d(d\! + \! 2)\beta^{12}}{a_\beta^2b_\beta^2} 
\right]
\nonumber\\
&\quad
-
a_\beta^{-3d/2}
\left[
8+\frac{12d\beta^4}{a_\beta^2} +\frac{8d\beta^6}{a_\beta^3}
+\frac{6d(d\! + \! 2)\beta^8}{a_\beta^4} +
\frac{d(d\! +\! 2)(d\! +\! 8)\beta^{12}}{a_\beta^6} \right].
\end{align*}

For the fourth cumulant, Proposition~\ref{prop:cumulants} gives
\begin{align*}
 \kappa_4(d,\beta)
 &=  48  \int_{(\R^d)^4}  K(s,t)K(t,u)K(u,v)K(v,s)
 \nonumber\\
 &\qquad\qquad\times  \varphi_\beta(s)\varphi_\beta(t)
 \varphi_\beta(u)\varphi_\beta(v)  \,\D s\,\D t\,\D u\, \D v.
\end{align*}

The following result is the main analytic contribution of this paper.

\medskip

\begin{theorem}\label{thm:kappa4}
For every $d\geq 1$ and $\beta>0$, the fourth cumulant of
$T_\beta(d)$ is
\begin{align*}
\kappa_4(d,\beta) = 48\Bigl\{ &D_4^{-d/2} -4D_3^{-d/2}P_3 +4D_{2a}^{-d/2}P_{2a}
\nonumber\\
&+2D_{2o}^{-d/2}P_{2o} -4D_1^{-d/2}P_1 +D_0^{-d/2}P_0 \Bigr\},
\end{align*}
where the quantities $D_4,D_3,D_{2a},D_{2o},D_1,D_0$ and
$P_3,P_{2a},P_{2o},P_1,P_0$ are given explicitly in
Appendix~\ref{sec:appendix_kappa4}. Consequently,
\[
\mu_4(d,\beta) = \E\left[ \big(T_\beta(d)- \E T_\beta(d)\big)^4\right] = \kappa_4(d,\beta) + 3 \kappa_2(d,\beta)^2.
\]
\end{theorem}

\medskip

\begin{remark}\label{rem:check}
As an independent check on Theorem~\ref{thm:kappa4}, setting
$d=1$ and $\beta=1$ gives
\[
 \kappa_4(1,1)  =  0.001654655083\ldots
\]
and
\[
 \mu_4(1,1)  =  0.002351088558535900\ldots.
\]
These values agree with those obtained independently for the
Epps--Pulley statistic in \cite{henze1990}.
\end{remark}

\medskip

\begin{remark}\label{rem:spectral_power_sum}
By the spectral representation~\eqref{eq:chisquare_rep},
\[
  \kappa_4(d,\beta)
  =
  48\sum_{j\geq1}\lambda_j(\beta,d)^4.
\]
Hence, Theorem~\ref{thm:kappa4} also provides an explicit closed-form
expression for the fourth power sum of the complete BHEP spectrum.
The recent spectral characterization in \cite{ebneretal26}
determines all eigenvalues, including their multiplicities, but some
of them are specified through roots of transcendental equations or
of a Fredholm determinant. Thus, the closed-form expression obtained
here does not require a numerical spectral calculation.
\end{remark}

\section{Four-moment approximations and $p$-values}
\label{sec:approximations}

Put $\mu=\kappa_1$, $\sigma^2=\kappa_2$, and define the standardized skewness and kurtosis by
\begin{equation*}
  \gamma_1(d,\beta)   =   \frac{\kappa_3(d,\beta)}{\kappa_2(d,\beta)^{3/2}},
  \qquad   \beta_2(d,\beta)=3+\frac{\kappa_4(d,\beta)}{\kappa_2(d,\beta)^2}.
\end{equation*}
Thus, the first four moments required for moment matching are
available explicitly as functions of $d$ and $\beta$.

We consider two classical four-moment approximation schemes, namely
the Johnson and Pearson systems of distributions. Both systems were
used in the special case $d=1$, $\beta=1$ in \cite{henze1990}.
Our main interest here is in obtaining an analytic approximation to the whole
distribution function and hence to the upper-tail probability
associated with an observed value of the BHEP statistic.

\subsection{The Johnson system}\label{subsec:johnson}

Johnson \cite{johnson1949} introduced a system of distributions obtained
by transformations to a standard normal random variable. Let
$Z\sim N(0,1)$, and let $T$ denote a random variable whose distribution
is to be fitted to that of $T_\beta(d)$. Depending on the values of
skewness and kurtosis, $T$ belongs to the
lognormal, unbounded, or bounded part of the Johnson system.

On the lognormal boundary, the moment-matched Johnson distribution belongs to the $S_L$ family and coincides with the three-parameter lognormal approximation in Section~\ref{sec:lognormal}. Its upper-tail   probability is given by the corresponding formula there, including the value 1 below its lower support endpoint. 

For the unbounded system $S_U$,
\[
Z=\gamma+\delta\,\operatorname{arsinh} \left(\frac{T-\xi}{\lambda}\right),
\]
whereas for the bounded system $S_B$,
\[
Z=\gamma+\delta\log \left(\frac{T-\xi}{\xi+\lambda-T}\right).
\]
Here, $\gamma\in\mathbb R$, $\delta>0$, $\lambda>0$, and $\xi\in\mathbb R$
are determined by matching the first four moments of $T$ with those of
$T_\beta(d)$. A classical algorithm for fitting Johnson distributions by moments
is given by Hill et al.~\cite{hill1985}.

Let $F_{J,d,\beta}$ denote the distribution function of the
moment-matched Johnson distribution. For an observed value $t$,
we define the Johnson approximation to the upper-tail probability by
\begin{equation*}
  \widehat p_J(t;d,\beta)   =   1-F_{J,d,\beta}(t).
\end{equation*}

Writing $\Phi$ for the distribution function of the standard normal
distribution, for an $S_U$ fit this becomes
\[
\widehat p_J(t;d,\beta)
=
1-\Phi\left[
\gamma+\delta\,\operatorname{arsinh}
\left(\frac{t-\xi}{\lambda}\right)
\right],
\qquad t\in\mathbb R,
\]
whereas an $S_B$ fit gives
\[
\widehat p_J(t;d,\beta)
=
\begin{cases}
1, & t\le \xi,\\[1mm]
1-\Phi\left[
\gamma+\delta\log\left(
\frac{t-\xi}{\xi+\lambda-t}
\right)
\right],
& \xi<t<\xi+\lambda,\\[1mm]
0, & t\ge \xi+\lambda.
\end{cases}
\]

Thus, once $T_{n,\beta}$ has been computed, the additional
computational effort required for obtaining $\widehat p_J$ is
negligible.

\medskip

\begin{remark}
For $d=1$ and $\beta=1$, the construction reduces to that studied
in \cite{henze1990}. In this case the moment-matched Johnson
distribution belongs to the bounded system $S_B$.
\end{remark}

\medskip

\begin{remark}\label{rem:bounded}
The true limiting distribution in~\eqref{eq:chisquare_rep} is
unbounded above, whereas an $S_B$ approximation has bounded support.
This discrepancy may become relevant for very small upper-tail
probabilities. We therefore pay particular attention to the extreme
upper tail in the numerical study below.
\end{remark}

\subsection{The Pearson system}\label{subsec:pearson}

A second four-moment approximation is obtained from the Pearson
system of distributions; see, e.g., \cite{eldertonjohnson69}. A Pearson density $f$ satisfies a differential equation of the form
\begin{equation*}
  \frac{f'(x)}{f(x)}
  =
  -\,\frac{x-a}
          {c_0+c_1x+c_2x^2},
\end{equation*}
with constants determined, up to the usual location and scale
standardization, by the first four moments.

For the moment combinations occurring in our numerical study, the
moment-matched Pearson distribution is either of Type VI or of
Type IV. We use the following parametrizations.

For a Pearson Type VI distribution, let
\[
 T=\eta+\tau Y,\qquad \tau>0,
\]
where $Y$ has a beta-prime distribution with parameters $p>0$ and
$q>4$, that is,
\[
 f_Y(y)  =  \frac{y^{p-1}}{B(p,q)(1+y)^{p+q}},  \qquad y>0.
\]
Its mean and variance are
\[
 \E(Y)=\frac{p}{q-1},
 \qquad
 \operatorname{Var}(Y)
 =  \frac{p(p+q-1)}{(q-2)(q-1)^2}.
\]
The standardized skewness and kurtosis are
\begin{align}
 \gamma_1^{\rm VI}  &=  \frac{2(2p+q-1)\sqrt{q-2}}{(q-3)\sqrt{p(p+q-1)}},
 \label{eq:pearsonVI_skew}
 \\
 \beta_2^{\rm VI}  &=  3+ 
 \frac{6\left\{ p(p+q-1)(5q-11)+(q-1)^2(q-2)\right\}}{p(p+q-1)(q-3)(q-4)}.
 \label{eq:pearsonVI_kurt}
\end{align}
Thus, $p$ and $q$ are determined by matching
\eqref{eq:pearsonVI_skew} and \eqref{eq:pearsonVI_kurt} to
$\gamma_1(d,\beta)$ and $\beta_2(d,\beta)$, respectively. Once
$p$ and $q$ have been obtained, location and scale are given by
\[
 \tau  \ =  \ \frac{\sigma}{\sqrt{\operatorname{Var}(Y)}}, \qquad 
 \eta  \ =  \  \mu-\tau\frac{p}{q-1}.
\]

For a Pearson Type IV distribution, we write again
\[
 T=\eta+\tau Y,\qquad \tau>0,
\]
where $Y$ has density
\begin{equation*}
 f_Y(y)  =  c_{m,\nu}  (1+y^2)^{-m}  \exp\{-\nu\arctan y\},
 \qquad y\in\mathbb R,
\end{equation*}
with $m>5/2$, $\nu\in\mathbb R$, and
\[
 c_{m,\nu}  =  
 \frac{2^{2m-2}\left|\Gamma\left(m+\frac{{\rm i}\nu}{2}\right)\right|^2}
 {\pi\Gamma(2m-1)}.
\]
Put $A_{m,\nu}=4(m-1)^2+\nu^2$. Then
\[
\E(Y)=-\frac{\nu}{2(m-1)}, \qquad 
 \operatorname{Var}(Y)  =  \frac{A_{m,\nu}}{4(m-1)^2(2m-3)}.
\]
Moreover,
\begin{align}
 \gamma_1^{\rm IV}
 &=
 -\frac{2\nu\sqrt{2m-3}}{(m-2)\sqrt{A_{m,\nu}}},
 \label{eq:pearsonIV_skew}
 \\
 \beta_2^{\rm IV}
 &=
 \frac{3(2m-3)\left\{4(m-2)(m-1)^2+(m+2)\nu^2\right\}}{(m-2)(2m-5)A_{m,\nu}}.
 \label{eq:pearsonIV_kurt}
\end{align}
The parameters $m$ and $\nu$ are obtained by matching
\eqref{eq:pearsonIV_skew} and \eqref{eq:pearsonIV_kurt} to
$\gamma_1(d,\beta)$ and $\beta_2(d,\beta)$. The remaining
parameters are then
\[
 \tau  \ = \  \frac{\sigma}
 {\sqrt{\operatorname{Var}(Y)}}, \qquad
\eta  \ =
\  \mu+\tau\frac{\nu}{2(m-1)}.
\]

Let $F_{P,d,\beta}$ denote the distribution function of this
moment-matched Pearson distribution. We define the Pearson
approximation to the upper-tail probability by
\[
  \widehat p_P(t;d,\beta)
  =
  1-F_{P,d,\beta}(t).
\]

In the special case $d=1$ and $\beta=1$, the corresponding Pearson
distribution is of Type~VI; see \cite{henze1990}. In contrast to the
bounded Johnson $S_B$ approximation arising in this case, a
Pearson Type~VI distribution has an unbounded right tail. This
feature may be advantageous when very small $p$-values are of
interest.

The numerical comparison below is intended to determine which of
the two four-moment systems provides the more accurate approximation
for a given combination of $d$ and $\beta$.

\subsection{Lognormal benchmark approximations}
\label{sec:lognormal}

For comparison, we also consider the two- and three-parameter
lognormal approximations used in \cite{henzewagner1997}, denoted by
$L_2$ and $L_3$, respectively. The former matches the first two
moments, whereas the latter matches the first three moments of
$T_\beta(d)$.

For the two-parameter lognormal approximation, put
\begin{equation*}
 s_2^2
 =
 \log\left(1+\frac{\sigma^2}{\mu^2}\right),
 \qquad
 m_2
 =
 \log\mu-\frac{s_2^2}{2}.
\end{equation*}
Then $T_\beta(d)  \approx  \exp(m_2+s_2Z)$,
 where $Z\sim N(0,1)$,
has expectation $\mu$ and variance $\sigma^2$.  The corresponding
approximation to the upper-tail probability is
\begin{equation*}
 \widehat p_{L_2}(t;d,\beta)
 =
 1-\Phi\left(
 \frac{\log t-m_2}{s_2}
 \right),
 \qquad t>0.
\end{equation*}
For the three-parameter lognormal approximation, let $a>1$ be the
unique solution of
\begin{equation*}
 \gamma_1(d,\beta)
 =
 (a+2)\sqrt{a-1}.
\end{equation*}
Put
\begin{equation*}
 s_3=\sqrt{\log a},
 \quad \
 \ell=\frac{\sigma}{\sqrt{a-1}},
 \quad \
 \eta=\mu-\ell,
 \quad \
 m_3=\log\ell-\frac12\log a.
\end{equation*}
Then $T_\beta(d) \approx \eta+\exp(m_3+s_3Z)$, where
 $Z\sim N(0,1)$, has the same expectation, variance and standardized skewness as
$T_\beta(d)$, and therefore matches its first three moments.

The resulting upper-tail approximation is
\begin{equation*}
 \widehat p_{L_3}(t;d,\beta)
 =  1-\Phi\left( \frac{\log(t-\eta)-m_3}{s_3}\right),
 \qquad t>\eta. 
\end{equation*}
For $t\leq\eta$, we set $\widehat p_{L_3}(t;d,\beta)=1$.

The two lognormal approximations are computationally as inexpensive
as the Johnson and Pearson approximations.  They provide useful
benchmarks for assessing the gain obtained by matching the fourth
moment.

\section{Accuracy of the limiting approximations}
\label{sec:limit_accuracy}

The main purpose of the four-moment approximations is to obtain
accurate upper-tail probabilities and, hence, accurate approximate
$p$-values. We therefore assess their performance primarily in terms
of upper-tail probabilities rather than by comparing only a small
number of quantiles.

For $M\in\{J,P,L_2,L_3\}$, let $q_{1-\alpha}^{(M)}(d,\beta)$
denote the $(1-\alpha)$-quantile of the corresponding approximating
distribution. If $T_\beta^{\rm ref}(d)$ denotes a random variable with the highly accurate
reference distribution described in Section~\ref{sec:spectral_reference},
define
\begin{equation*}
R_\alpha^{(M)}(d,\beta)
=
\frac{\mathbb P\left\{
T_\beta^{\rm ref}(d)>q_{1-\alpha}^{(M)}(d,\beta)\right\}}{\alpha}-1,
\quad \ M\in\{J,P,L_2,L_3\}.
\end{equation*}
Thus, $R_\alpha^{(M)}(d,\beta)=0$ corresponds to exact upper-tail
calibration. Positive values indicate liberal behavior at nominal
level $\alpha$, whereas negative values indicate conservative
behavior.

We consider the nominal probabilities
$\alpha\in\{0.10,0.05,0.025,0.01,0.005,$ $0.001\}$,
the dimensions $d\in\{1,2,3,5,10\}$,
and the smoothing parameters $\beta\in\{0.1,0.25,0.5,1,2,3,5\}$.

The Johnson and Pearson four-moment approximations are compared with
the two- and three-parameter lognormal approximations described in
Section~\ref{sec:lognormal}.

\subsection{Moment characteristics}\label{subsec:moment_characteristics}

Table~\ref{tab:moments} displays the standardized skewness and
kurtosis, together with the resulting Johnson and Pearson types, for
selected combinations of $d$ and $\beta$.

\begin{table}[ht]
\centering
\caption{Standardized moments and fitted four-moment distributions
for selected values of $d$ and $\beta$.}
\label{tab:moments}
\smallskip
\begin{tabular}{cccccc}
\toprule
$d$ & $\beta$ & $\gamma_1$ & $\beta_2$
& Johnson type & Pearson type\\
\midrule
1  & 1   & 2.12870 & 10.12769 & $S_B$ & VI\\
2  & 0.5 & 1.56986 &  7.05206 & $S_B$ & VI\\
2  & 1   & 1.24129 &  5.56256 & $S_B$ & VI\\
3  & 1   & 0.86437 &  4.31073 & $S_B$ & VI\\
5  & 1   & 0.51336 &  3.51199 & $S_U$ & IV\\
10 & 3   & 0.00548 &  3.00013 & $S_U$ & IV\\
\bottomrule
\end{tabular}
\end{table}

The parameters of the moment-matched Johnson distributions for the
six representative cases considered in Table~\ref{tab:moments} are
given in Table~\ref{tab:johnson_parameters}. As a check, the values
for $d=1$ and $\beta=1$ agree with those obtained in
\cite{henze1990}.

The parameter values in Tables~\ref{tab:johnson_parameters} 
and~\ref{tab:pearson_parameters} are rounded to six decimal
places for display. All numerical accuracy calculations reported below
were carried out using the unrounded fitted parameters.

\begin{table}[ht]
\centering
\caption{Parameters of the moment-matched Johnson distributions for
selected values of $d$ and $\beta$.}
\label{tab:johnson_parameters}
\smallskip
\begin{tabular}{cccrrrr}
\toprule
$d$ & $\beta$ & Type & $\gamma$ & $\delta$ & $\xi$ & $\lambda$\\
\midrule
 1 & 1   & $S_B$ &   3.552954 &   1.230622 & -0.020682 &  2.266644\\
 2 & 0.5 & $S_B$ &   4.551631 &   1.676188 & -0.001720 &  0.543393\\
 2 & 1   & $S_B$ &   5.456845 &   2.091337 & -0.006193 &  4.038107\\
 3 & 1   & $S_B$ &  10.626002 &   3.292173 &  0.031818 & 10.646894\\
 5 & 1   & $S_U$ &  -6.794478 &   5.183662 &  0.332850 &  0.211960\\
10 & 3   & $S_U$ & -83.358634 & 208.125110 &  0.986552 &  0.032676\\
\bottomrule
\end{tabular}
\end{table}

The corresponding parameters of the moment-matched Pearson
distributions are given in Table~\ref{tab:pearson_parameters}.
For the first four cases the fitted distribution is of Type~VI,
whereas for the last two cases it is of Type~IV. The selected cases
therefore cover both bounded and unbounded Johnson fits and both
Pearson types occurring in the range considered here.

\begin{table}[ht]
\centering
\caption{Parameters of the moment-matched Pearson distributions
for selected values of $d$ and $\beta$.}
\label{tab:pearson_parameters}
\smallskip
\begin{tabular}{ccccccc}
\toprule
$d$ & $\beta$ & Type & $p$ or $m$ & $q$ or $\nu$ & $\eta$ & $\tau$\\
\midrule
1  & 1   & VI & 1.020945 & 42.699009 & 0.012239 & 4.972110\\
2  & 0.5 & VI & 2.200202 & 31.288258 & 0.005780 & 0.430292\\
2  & 1   & VI & 3.735855 & 37.064855 & 0.048217 & 2.394881\\
3  & 1   & VI & 10.663428 & 41.472364 & 0.095231 & 1.364437\\
5  & 1   & IV & 29.902700 & -176.357273 & 0.106502 & 0.195909\\
10 & 3   & IV & 34352.547122 & -26451.585712 & 0.984004 & 0.041526\\
\bottomrule
\end{tabular}
\end{table}

\subsection{Spectral reference distribution}
\label{sec:spectral_reference}

To assess the accuracy of the four analytic approximations, we use
the complete spectral characterization of the limiting BHEP
operator obtained in \cite{ebneretal26}.  Recall that
\[
T_\beta(d)
\stackrel{\mathcal D}{=}
\sum_{j\geq 1}\lambda_j(\beta,d)N_j^2,
\]
where the positive eigenvalues are repeated according to their
multiplicities.

Following the decomposition in \cite{ebneretal26}, the spectrum consists of
the unchanged geometric eigenvalues together with the eigenvalues
arising from the radial, degree-one and degree-two sectors.
For the numerical calculations, the affected block representations
are truncated at dimension $K=180$, while the unchanged geometric
part of the spectrum is included up to total Hermite degree $N=250$.
These cutoffs are chosen sufficiently large that further increases
have no effect on the reported upper-tail probabilities, and the
first four spectral cumulants agree with their closed-form
counterparts to substantially more digits than those reported below.

A particularly useful accuracy check is provided by the cumulant
identities
\begin{equation}
\kappa_r(d,\beta)
=
2^{r-1}(r-1)!
\sum_{j\geq1}\lambda_j(\beta,d)^r, \qquad r=1,2,3,4.
\label{eq:spectral_cumulant_check}
\end{equation}
Thus, in addition to the trace identity, the closed-form expression
for the fourth cumulant obtained in Theorem~\ref{thm:kappa4}
provides an independent high-order check on the numerical spectral
calculation.

If $\lambda_1,\ldots,\lambda_m$ denote the retained eigenvalues,
including multiplicities, we use
\[
T_{\beta,m}^{\rm ref}(d)
=
\sum_{j=1}^m \lambda_j N_j^2
\]
as a numerical reference distribution. For fixed $m$, its
distribution function is evaluated by Imhof's inversion method for
quadratic forms in normal variables; see \cite{imhof61}. The
truncation level $m$ is increased until both the quantities in
\eqref{eq:spectral_cumulant_check} and the resulting upper-tail
probabilities are stable to substantially more digits than those
reported below.

For example, for $d=2$ and $\beta=1$, the relative discrepancies
between the cumulants obtained from the retained eigenvalues and
their closed-form values are
\[
6.6\times10^{-15},\quad
2.2\times10^{-15},\quad
3.9\times10^{-14},\quad
4.5\times10^{-13},
\]
for the first, second, third, and fourth cumulant, respectively.

This also provides a numerical check of the new expression for
$\kappa_4(d,\beta)$.

\subsection{Upper-tail accuracy}
\label{subsec:tail_accuracy}

For each combination of $d$, $\beta$, and $\alpha$, we compare the
nominal upper-tail probability $\alpha$ with the probability obtained
under the reference limiting distribution when the critical value is
taken from one of the approximating distributions.

In addition to the relative calibration errors
$R_\alpha^{(J)}(d,\beta)$ and $R_\alpha^{(P)}(d,\beta)$, we report the
corresponding errors for the two- and three-parameter lognormal
approximations.

\begin{table}[ht]
\centering
\caption{Relative errors of analytic approximations in the upper tail.
The entries are $R_\alpha^{(M)}(d,\beta)$.}
\label{tab:tail_errors}
\smallskip
\begin{tabular}{ccc rrcc}
\toprule
$d$ & $\beta$ & $\alpha$
& Johnson & Pearson
& $L_2$ & $L_3$\\
\midrule
2 & 1 & 0.100 &  0.00158 & -0.00305 &  0.03859 &  0.01904\\
2 & 1 & 0.050 & -0.00028 &  0.00107 &  0.01128 &  0.02822\\
2 & 1 & 0.025 & -0.00281 &  0.00649 & -0.03968 &  0.02850\\
2 & 1 & 0.010 & -0.00594 &  0.01277 & -0.13819 &  0.01130\\
2 & 1 & 0.005 & -0.00727 &  0.01486 & -0.23053 & -0.01669\\
2 & 1 & 0.001 & -0.00294 &  0.00291 & -0.47091 & -0.13174\\
\bottomrule
\end{tabular}
\end{table}

For the pilot case $d=2$, $\beta=1$, both four-moment
approximations are remarkably accurate throughout the range of
nominal probabilities considered. The relative calibration error
of the Johnson approximation remains below one percent in absolute
value, while that of the Pearson approximation remains below about
$1.5\%$. In contrast, the two-parameter lognormal approximation
deteriorates markedly in the extreme upper tail. The three-parameter
lognormal approximation performs considerably better, but is also
less accurate than the four-moment approximations at the smallest
nominal probabilities considered.

To assess whether this behavior is representative, we next consider
all six combinations of $d$ and $\beta$ listed in Table~\ref{tab:moments}. For each approximation, Table~\ref{tab:max_tail_errors} reports the maximum absolute relative calibration error
\[
\max_{\alpha\in\{0.10,0.05,0.025,0.01,0.005,0.001\}}
\left|R_\alpha^{(M)}(d,\beta)\right|.
\]

\begin{table}[ht]
\centering
\caption{Maximum absolute relative calibration errors over
$\alpha\in\{0.10,0.05,0.025,0.01,0.005,0.001\}$ for the
representative cases.}
\label{tab:max_tail_errors}
\smallskip
\begin{tabular}{ccrrcc}
\toprule
$d$ & $\beta$ & Johnson & Pearson
& $L_2$ & $L_3$\\
\midrule
 1 & 1   & 0.05486 & 0.01405 & 0.79237 & 0.20432\\
 2 & 0.5 & 0.01702 & 0.03148 & 0.58901 & 0.16204\\
 2 & 1   & 0.00727 & 0.01486 & 0.47091 & 0.13174\\
 3 & 1   & 0.00433 & 0.00905 & 0.02018 & 0.03951\\
 5 & 1   & 0.00366 & 0.00283 & 0.74417 & 0.05699\\
10 & 3   & $4.37\times10^{-6}$ & $4.11\times10^{-6}$
           & 0.02425 & $2.27\times10^{-4}$\\
\bottomrule
\end{tabular}
\end{table}

The results show that both four-moment approximations provide very
accurate upper-tail calibration for the representative cases.
With the exception of the bounded Johnson $S_B$ approximation for
$d=\beta=1$, the maximum absolute relative error is below about
$3.2\%$, and in most cases it is substantially smaller. For
$d=\beta=1$, the Pearson Type~VI approximation is clearly more
accurate in the extreme upper tail, which is consistent with the
unbounded support of the true limiting distribution. The
two-parameter lognormal approximation may deteriorate severely in
the extreme tail, whereas the three-parameter lognormal
approximation performs considerably better, but is generally less
accurate than the four-moment fits.

To summarize the performance of the two four-moment approximations
over the full grid, put, for $M \in \{J,P,L_2,L_3\}$,
\[
E_M(d,\beta)
=
\max_{\alpha\in\{0.10,0.05,0.025,0.01,0.005,0.001\}}
\left|R_\alpha^{(M)}(d,\beta)\right|.
\]
Table~\ref{tab:full_grid_best} reports, for each pair $(d,\beta)$,
the better of the Johnson and Pearson approximations. An entry
$J(x)$ or $P(x)$ means that the corresponding approximation has the
smaller value of $E_M(d,\beta)$, with $100E_M(d,\beta)=x$ expressed
as a percentage.

\begin{table}[ht]
\centering
\caption{Better four-moment approximation over the full grid.
The number in parentheses is the maximum absolute relative
calibration error, in percent, over the six nominal levels.}
\label{tab:full_grid_best}
\medskip
\small
\setlength{\tabcolsep}{3pt}
\renewcommand{\arraystretch}{1.08}

\begin{tabular}{c@{\hspace{5pt}}ccccccc}
\toprule
& \multicolumn{7}{c}{$\beta$}\\
\cmidrule(lr){2-8}
$d$
& $0.1$ & $0.25$ & $0.5$ & $1$ & $2$ & $3$ & $5$\\
\midrule
 1 & $P(0.03)$ & $P(0.67)$ & $P(2.77)$ & $P(1.40)$
   & $P(2.28)$ & $J(1.58)$ & $J(0.60)$\\
 2 & $P(1.93)$ & $P(2.27)$ & $J(1.70)$ & $J(0.73)$
   & $J(0.34)$ & $J(0.33)$ & $J(0.08)$\\
 3 & $J(1.20)$ & $J(1.23)$ & $J(1.05)$ & $J(0.43)$
   & $J(0.17)$ & $P(0.05)$ & $P(0.15)$\\
 5 & $J(1.17)$ & $J(1.26)$ & $J(1.09)$ & $P(0.28)$
   & $P(0.35)$ & $P(0.17)$ & $P(0.01)$\\
10 & $J(1.19)$ & $J(1.05)$ & $J(0.29)$ & $P(0.57)$
   & $P(0.04)$ & $P(<.001)$ & $J(<.001)$\\
\bottomrule
\end{tabular}
\end{table}

Table~\ref{tab:full_grid_best} shows that neither four-moment
approximation uniformly dominates the other. The Johnson
approximation has the smaller maximum calibration error in 19 of
the 35 parameter combinations, whereas the Pearson approximation
is preferable in the remaining 16 cases. More importantly, for
every pair $(d,\beta)$ in the grid, at least one of the two
four-moment approximations has a maximum absolute relative
calibration error below $2.8\%$. More precisely,
\[
\max_{(d,\beta)}
\min\{E_J(d,\beta),E_P(d,\beta)\}
\approx 0.0277,
\]
the maximum being attained for $d=1$ and $\beta=0.5$.

The full-grid comparison with the lognormal benchmarks is much less
favorable to the latter. Over all 35 combinations of $d$ and
$\beta$, the largest values of $E_M(d,\beta)$ are approximately
$0.1089$ for the Johnson approximation and $0.0315$ for the Pearson
approximation, compared with $1.2314$ for the two-parameter
lognormal approximation and $0.2269$ for the three-parameter
lognormal approximation. Thus, matching the fourth moment leads to
a substantial improvement in worst-case upper-tail calibration.

If a single approximation is to be used throughout the parameter
range considered here, the Pearson approximation exhibits the more
uniform worst-case behavior. On the other hand, allowing the choice
between the Johnson and Pearson fits according to $(d,\beta)$ yields
a highly accurate approximation over the entire grid.

\section{Finite-sample $p$-values} \label{sec:finite_sample}

An accurate approximation to the limiting distribution does not
necessarily imply an equally accurate approximation for finite sample
sizes. We therefore investigate directly the finite-sample behavior of
the Johnson and Pearson $p$-value approximations.

Owing to affine invariance, it suffices to generate samples from
$N_d(0,I_d)$, the $d$-variate standard normal distribution. For each simulated sample we calculate the standardized
residuals $Y_1,\ldots,Y_n$ and evaluate the BHEP statistic in the
closed form
\[
\begin{split}
T_{n,\beta}
={}&
\frac{1}{n}
\sum_{j,k=1}^n
\exp\left\{
-\frac{\beta^2}{2}\|Y_j-Y_k\|^2
\right\}\\
&{}
-2(1+\beta^2)^{-d/2}
\sum_{j=1}^n
\exp\left\{
-\frac{\beta^2}{2(1+\beta^2)}
\|Y_j\|^2
\right\}\\
&{}
+n(1+2\beta^2)^{-d/2}.
\end{split}
\]
For each realization, the corresponding approximate $p$-values
\[
\widehat p_J = \widehat p_J(T_{n,\beta};d,\beta),
\qquad \widehat p_P = \widehat p_P(T_{n,\beta};d,\beta)
\]
are computed from the moment-matched Johnson and Pearson
distributions described in Section~\ref{sec:approximations}. The simulation study covers $n\in\{20,50,100,200,500\}$, $d\in\{1,2,3,5,10\}$, and $\beta\in\{0.1,0.25,0.5,1,2,3,5\}$.

For each combination of $n$ and $d$, the same simulated samples are
used for all seven values of $\beta$. We use $50\,000$ Monte Carlo
replications for $n\leq200$ and $20\,000$ replications for $n=500$.
For a nominal level $\alpha$, the finite-sample rejection probability
is estimated by
\[
\widehat\alpha_{M,n}(d,\beta)
=
\frac{1}{B_n} \sum_{r=1}^{B_n} \mathbf{1}
\left\{\widehat p_{M,r}\leq\alpha \right\}, \qquad M\in\{J,P\},
\]
where $B_n=50\,000$ for $n\leq200$ and $B_{500}=20\,000$.
For orientation, at nominal level $0.05$ the binomial Monte Carlo
standard error is approximately $0.0010$ for $50\,000$ replications
and $0.0015$ for $20\,000$ replications.

Table~\ref{tab:finite_005} reports empirical null rejection
probabilities at $\alpha=0.05$ for six
representative cases used in Section~\ref{sec:limit_accuracy}.
Each entry gives Johnson/Pearson.
Thus, for each simulated sample under the normal null hypothesis,
we reject whenever the approximate Johnson or Pearson $p$-value,
respectively, does not exceed the nominal level $\alpha$.

\begin{table}[ht]
\centering
\caption{Empirical rejection probabilities at nominal level
$\alpha=0.05$ for selected values of $d$ and $\beta$.
Each entry is Johnson/Pearson.}
\label{tab:finite_005}

\setlength{\tabcolsep}{3pt}
\renewcommand{\arraystretch}{1.08}
\smallskip
\begin{tabular}{ccccccc}
\toprule
& & \multicolumn{5}{c}{$n$}\\
\cmidrule(lr){3-7}
$d$ & $\beta$
& $20$ & $50$ & $100$ & $200$ & $500$\\
\midrule
1 & 1
& .0479/.0469
& .0488/.0478
& .0495/.0484
& .0516/.0507
& .0511/.0501\\

2 & 0.5
& .0363/.0362
& .0457/.0455
& .0486/.0484
& .0495/.0494
& .0503/.0502\\

2 & 1
& .0427/.0428
& .0475/.0475
& .0476/.0476
& .0496/.0497
& .0482/.0483\\

3 & 1
& .0399/.0400
& .0458/.0459
& .0476/.0477
& .0482/.0483
& .0489/.0490\\

5 & 1
& .0399/.0399
& .0469/.0469
& .0485/.0485
& .0492/.0492
& .0529/.0530\\

10 & 3
& .0005/.0005
& .0065/.0065
& .0166/.0166
& .0341/.0341
& .0551/.0551\\
\bottomrule
\end{tabular}
\end{table}

For the first five representative cases, the empirical rejection
probabilities approach their limiting values quite rapidly, and these
limiting values are close to the nominal level.
 The finite-sample tests tend to be slightly
conservative for small $n$, but for most of these parameter
combinations the empirical rejection probabilities are already close
to $0.05$ for $n=50$ or $n=100$. In the classical case $d=\beta=1$,
the finite-sample calibration is remarkably accurate even for
$n=20$.
The behavior can, however, be quite different in higher dimensions.
For $d=10$ and $\beta=3$, for example, the $5\%$ rejection
probability increases from approximately $0.0005$ at $n=20$ to
$0.0341$ at $n=200$ and $0.0551$ at $n=500$. The convergence in the
more extreme upper tail can be still slower. At nominal level
$\alpha=0.01$, the corresponding Pearson rejection probabilities for
$n=20,50,100,200,500$ are, respectively,
\[
0.0004,\quad
0.0046,\quad
0.0115,\quad
0.0212,\quad
0.0321.
\]
The Johnson values agree with these numbers to the displayed
precision. The approach to the limiting rejection probability need not be
monotone over the sample sizes considered. For example, in this case
the empirical rejection probability exceeds the nominal level for
$n\geq 100$ and reaches $0.0321$ at $n=500$. 

The full grid confirms that this slow convergence is not restricted
to a single parameter combination. In particular, for $d=10$ and
$\beta=5$, the empirical rejection probability at nominal level
$0.05$ is only about $0.0104$ even for $n=500$. Small values of
$\beta$ may also lead to slow convergence in high dimension; for
$d=10$ and $\beta=0.1$, the corresponding rejection probabilities
at level $0.05$ are approximately
\[
0.0000,\quad
0.0139,\quad
0.0313,\quad
0.0409,\quad
0.0461
\]
for $n=20,50,100,200,500$, respectively.

Nevertheless, the overall finite-sample behavior improves markedly
with increasing sample size. At $n=500$, 34 of the 35 combinations
of $(d,\beta)$ have empirical rejection probabilities between
$0.04$ and $0.06$ at nominal level $0.05$ for both four-moment
approximations; the only exception is $(d,\beta)=(10,5)$.
At nominal level $0.01$, 32 of the 35 combinations lie between
$0.005$ and $0.015$. The remaining cases are
$(d,\beta)=(5,5)$, $(10,2)$ and $(10,3)$, with empirical Pearson
rejection probabilities $0.0166$, $0.0164$ and $0.0321$,
respectively.

An important feature of the simulation results is that the Johnson
and Pearson finite-sample rejection probabilities are generally very
close. At nominal level $0.05$, their difference never exceeds about
$0.0026$ over the complete simulation grid, and at level $0.01$ it
never exceeds about $0.0006$. Thus, the more substantial
finite-sample discrepancies observed for some combinations of $d$,
$\beta$, and $n$ cannot primarily be attributed to the choice of the
four-moment approximation. Rather, they reflect the rate at which
the finite-sample distribution of $T_{n,\beta}$ approaches its
limiting null distribution.

In summary, the four-moment approximations lead to well-calibrated
finite-sample tests for a broad range of practically relevant
parameter combinations, often already for moderate sample sizes.
The simulation results also show, however, that finite-sample
calibration varies considerably with $d$ and $\beta$.
In particular, larger dimensions combined with smoothing parameters
far from the central range may require substantially larger sample
sizes before asymptotic $p$-values can be used with high accuracy.

\section{Discussion}\label{sec:discussion}

The main theoretical contribution of this paper is an explicit
closed-form expression for the fourth cumulant of the limiting null
distribution of the BHEP statistic for arbitrary dimension $d\geq1$
and smoothing parameter $\beta>0$. Together with the previously known
first three cumulants, this makes four-moment fitting available for the
whole BHEP family.

From a practical point of view, the main benefit is that the resulting
Johnson and Pearson approximations provide approximations to the entire
limiting null distribution rather than only to a collection of critical
values. Consequently, once $d$, $\beta$, and the observed value of the
BHEP statistic are given, approximate $p$-values can be obtained
essentially instantaneously. No simulation and no numerical calculation
of the eigenvalues of the covariance operator are required at the
application stage. The recently obtained complete spectrum is used here
only to construct highly accurate reference distributions against which
the moment-based approximations can be assessed.

The numerical results show that both four-moment approximations are
highly accurate over a broad range of dimensions, smoothing parameters,
and upper-tail probabilities. Neither the Johnson nor the Pearson
approximation uniformly dominates the other: Johnson has the smaller
maximum calibration error in 19 of the 35 parameter combinations
considered, whereas Pearson is preferable in the remaining 16.
For every pair $(d,\beta)$ in the grid, however, at least one of the
two approximations has a maximum absolute relative calibration error
below $2.8\%$. If a single approximation is to be used throughout the
parameter range, the Pearson approximation has the more uniform
worst-case behavior. Its maximum calibration error over the full grid
is about $3.15\%$, compared with about $10.9\%$ for the Johnson
approximation.

The comparison with the two- and three-parameter lognormal
approximations shows that the use of the fourth moment can lead to a
substantial improvement in upper-tail calibration. The improvement is
particularly pronounced in parameter regions in which the simpler
lognormal approximations have difficulty reproducing the extreme upper
tail. The bounded Johnson $S_B$ family also illustrates a structural
limitation of moment matching: although it may approximate the relevant
part of the distribution very accurately, its bounded support cannot
reproduce the unbounded right tail of the true limiting distribution.

The finite-sample simulations add an important qualification to these
asymptotic results. For many combinations of $d$ and $\beta$, the
resulting tests are already well calibrated for moderate sample sizes.

The finite-sample accuracy, however, varies considerably with $d$
and $\beta$. In particular, in higher dimensions and for smoothing
parameters far from the central range, substantially larger sample
sizes may be required before asymptotic $p$-values attain their
limiting calibration.
 The fact that the Johnson and Pearson rejection
probabilities are then nearly identical indicates that this phenomenon
is primarily due to the rate of convergence of $T_{n,\beta}$ to its
limiting distribution, rather than to the four-moment approximation
itself.

Thus, the fourth cumulant provides additional information about the limiting BHEP distribution and, more importantly for applications, enables a simple and effective analytic $p$-value approximation.  A natural next step would be to investigate whether finite-sample corrections can improve the calibration in those regions of $(d,\beta)$ for which
convergence to the limiting distribution is comparatively slow.

\section*{Software availability}
The Johnson and Pearson approximations developed in this paper are
implemented in version~1.4 of the \textsf{R} package \texttt{mnt}, see \cite{butschebner2026}, available on the \texttt{CRAN} archive.
The function \texttt{test.BHEP} allows users to choose between Monte
Carlo calibration and either of the two four-moment approximations.
For the latter, the fitted distribution matches the first four moments
of the limiting null distribution for the specified dimension and
smoothing parameter, providing approximate critical values and
$p$-values without repeated simulation. 

\appendix

\section{Closed form of the fourth cumulant}
\label{sec:appendix_kappa4}

Put $ b=\beta^2$, $g=1+2b$, and write
\[
 J_4(d,\beta)=\frac{\kappa_4(d,\beta)}{48}.
\]

To evaluate $J_4(d,\beta)$, we use a Gaussian generating function.
For $q=(q_1,q_2,q_3,q_4)\in[0,1]^4$, put $x_5=x_1$ and define
\begin{align}
\mathcal G(q)
&=
(2\pi b)^{-2d}
\int_{(\mathbb R^d)^4}
\exp\left\{
-\frac{g}{2b}\sum_{j=1}^4\|x_j\|^2
+\sum_{j=1}^4q_jx_j^\top x_{j+1}
\right\}
\,\D x_1\cdots \D x_4 .
\label{eq:Ggenerator}
\end{align}

\begin{lemma}\label{lem:gaussian_generator}
For $q\in[0,1]^4$, we have
\begin{equation*}
 \mathcal G(q)=\Delta(q)^{-d/2},
\end{equation*}
where
\begin{equation*}
 \Delta(q)
 =
 g^4-b^2g^2\sum_{j=1}^4q_j^2
 +b^4(q_1q_3-q_2q_4)^2.
\end{equation*}
\end{lemma}

\smallskip
\begin{proof}
For one coordinate, the quadratic form in
\eqref{eq:Ggenerator} can be written as
\[
 -\frac{1}{2b}x^\top C(q)x,
\]
where
\[
C(q)=
\begin{pmatrix}
g       & -bq_1 & 0       & -bq_4\\
-bq_1   & g     & -bq_2   & 0\\
0       & -bq_2 & g       & -bq_3\\
-bq_4   & 0     & -bq_3   & g
\end{pmatrix}.
\]
Since $q_j\in[0,1]$ and $g=1+2b$, the symmetric matrix
$C(q)$ is strictly diagonally dominant with positive diagonal
entries and is therefore positive definite. 
Hence, by the standard multivariate Gaussian integral,
\[
 \mathcal G(q)=\det C(q)^{-d/2}.
\]
A direct calculation gives
\[
 \det C(q)  =  g^4-b^2g^2(q_1^2+q_2^2+q_3^2+q_4^2)
 +b^4(q_1q_3-q_2q_4)^2,
\]
which proves the assertion.
\end{proof}

For later use, we introduce the following special values of
$\Delta(q)$:
\begin{align*}
D_4
&:=\Delta(1,1,1,1)
 =g^2(1+4b),
\\
D_3
&:=\Delta(0,1,1,1)
 =(1+3b+b^2)(1+5b+5b^2),
\\
D_{2a}
&:=\Delta(0,0,1,1)
 =g^2(1+4b+2b^2),
\\
D_{2o}
&:=\Delta(0,1,0,1)
 =(1+b)^2(1+3b)^2,
\\
D_1
&:=\Delta(0,0,0,1)
 =g^2(1+b)(1+3b),
\\
D_0
&:=\Delta(0,0,0,0)
 =g^4.
\end{align*}

Differentiation under the integral sign is justified by dominated
convergence, since $C(q)$ is uniformly positive definite for
$q\in[0,1]^4$.

For $j=1,\ldots,4$, introduce the differential operator
\begin{equation*}
 \mathcal L_j
 =
 1+\frac{\partial}{\partial q_j}
 +\frac12\frac{\partial^2}{\partial q_j^2}.
\end{equation*}
Since
\[
 \left.
 \mathcal L_j
 \exp\{q_jx_j^\top x_{j+1}\}
 \right|_{q_j=0}
 =
 1+x_j^\top x_{j+1}
 +\frac12(x_j^\top x_{j+1})^2,
\]
the polynomial part of the covariance kernel can be generated by
differentiation of $\mathcal G$.
More precisely, let $\mathcal I=\{1,2,3,4\}$,
and, for $E\subset\mathcal I$, define
$q^E=(q_1^E,\ldots,q_4^E)$ by $q_j^E=\mathbf 1\{j\in E\}$.
Here $E$ represents the set of factors in which the exponential part
\[
 \exp\left(-\frac{\|x_j-x_{j+1}\|^2}{2}\right)
\]
of the covariance kernel is selected. Expanding the product of the
four covariance kernels therefore gives
\begin{equation}
 J_4(d,\beta)  =  \sum_{E\subset\mathcal I}
 (-1)^{4-|E|}  \left.  \left(  \prod_{j\notin E}\mathcal L_j  \right)
 \mathcal G(q)  \right|_{q=q^E}. 
 \label{eq:J4subset}
\end{equation}

By the cyclic and reflection symmetries of the four-cycle, the
$16$ terms in \eqref{eq:J4subset} fall into six classes, denoted by
$4$, $3$, $2a$, $2o$, $1$,  and $0$, where $2a$ and $2o$ refer to two adjacent and two opposite
exponential factors, respectively. Representative terms are
\begin{align}
 \mathcal G(1,1,1,1)
 &=
 D_4^{-d/2},
\label{eq:class4}
\\
 \left.
 \mathcal L_1\mathcal G(q)
 \right|_{q=(0,1,1,1)}
 &=
 D_3^{-d/2}P_3,
\label{eq:class3}
\\
 \left.
 \mathcal L_1\mathcal L_2\mathcal G(q)
 \right|_{q=(0,0,1,1)}
 &=
 D_{2a}^{-d/2}P_{2a},
\label{eq:class2a}
\\
 \left.
 \mathcal L_1\mathcal L_3\mathcal G(q)
 \right|_{q=(0,1,0,1)}
 &=
 D_{2o}^{-d/2}P_{2o},
\label{eq:class2o}
\\
 \left.
 \mathcal L_1\mathcal L_2\mathcal L_3\mathcal G(q)
 \right|_{q=(0,0,0,1)}
 &=
 D_1^{-d/2}P_1,
\label{eq:class1}
\\
 \left.
 \mathcal L_1\mathcal L_2\mathcal L_3\mathcal L_4
 \mathcal G(q)
 \right|_{q=(0,0,0,0)}
 &=
 D_0^{-d/2}P_0.
\label{eq:class0}
\end{align}

There are, respectively,
\[
 1,\quad4,\quad4,\quad2,\quad4,\quad1
\]
terms of these six types. Taking the signs in
\eqref{eq:J4subset} into account therefore yields
\begin{equation}
\begin{split}
 J_4(d,\beta)
 &=
 D_4^{-d/2}
 -4D_3^{-d/2}P_3
 +4D_{2a}^{-d/2}P_{2a}
\\
 &\quad
 +2D_{2o}^{-d/2}P_{2o}
 -4D_1^{-d/2}P_1
 +D_0^{-d/2}P_0.
\end{split}
\label{eq:J4}
\end{equation}

The quantities $P_3,P_{2a},P_{2o},P_1$ and $P_0$ appearing in
\eqref{eq:class3}--\eqref{eq:class0} are obtained by differentiating
$\Delta(q)^{-d/2}$. Explicitly,
\begin{align*}
P_3
&=
1+
\frac{d b^2(1+4b+5b^2)}{2D_3}
+
\frac{d(d+2)b^8}{2D_3^2},
\\[1mm]
P_{2a}
&=
1+
\frac{d b^2g^2}{D_{2a}}
+
\frac{d(d+2)b^4
(1+8b+22b^2+24b^3+11b^4)}
{4D_{2a}^2},
\\[1mm]
P_{2o}
&=
1+
\frac{d b^2(2+8b+9b^2)}{2D_{2o}}
+
\frac{d(d+2)b^4
(1+8b+24b^2+32b^3+18b^4)}
{4D_{2o}^2},
\\[1mm]
P_1
&=
1+
\frac{3d b^2g^2}{2D_1}
+
\frac{3d(d+2)b^4
(1+8b+23b^2+28b^3+13b^4)}
{4D_1^2}
\nonumber\\
&\quad+
\frac{d(d+2)(d+4)b^6g^2}
{8D_1^2},
\\[1mm]
P_0
&=
1+\frac{2db^2}{g^2}
+
\frac{d(d+2)b^4
(6+48b+140b^2+176b^3+83b^4)}
{4g^8}
\nonumber\\
&\quad+
\frac{d(d+2)(d+4)b^6
(2+8b+7b^2)}
{4g^8}
+
\frac{d(d+2)(d+4)(d+6)b^8}
{16g^8}.
\end{align*}

Combining \eqref{eq:J4} with
$\kappa_4(d,\beta)=48J_4(d,\beta)$ 
gives the explicit closed-form expression for the fourth cumulant
stated in Theorem~\ref{thm:kappa4}. This completes the proof.

\medskip

\begin{remark}\label{rem:kappa4_check}
For $d=1$ and $\beta=1$, substitution into
\eqref{eq:J4} gives
\[
 \kappa_4(1,1)
 =
 0.001654655083\ldots,
\]
and hence
\[
 \mu_4(1,1)
 =
 \kappa_4(1,1)+3\kappa_2(1,1)^2
 =
 0.002351088558535900\ldots.
\]
Both values agree with the independent calculation in
\cite{henze1990}.
\end{remark}

\section*{Acknowledgement}

OpenAI's ChatGPT was used as an AI-assisted research tool in the
preparation of this work. In particular, the key formula underlying
the derivation of the fourth cumulant in Theorem~\ref{thm:kappa4} was first
suggested by ChatGPT and was subsequently developed, checked, and
independently verified by the authors. ChatGPT was also used for
assistance with exposition, \LaTeX{} formulation, and consistency
checks. The authors take full responsibility for all results and
conclusions in the paper.

\end{document}